\documentclass[12pt]{article}

\usepackage{cmap}  
\usepackage[T2A]{fontenc}
\usepackage[utf8]{inputenc}
\usepackage[english]{babel}
\usepackage[usenames]{color}
\usepackage{amsmath,amssymb,amsthm}
\usepackage[pdftex,colorlinks=true,linkcolor=blue,urlcolor=red,unicode=true,hyperfootnotes=false,bookmarksnumbered]{hyperref}
\usepackage{mathtools}
\usepackage{xcolor}

\newcommand{\eps}{\varepsilon}

\renewcommand{\le}{\leqslant}
\renewcommand{\ge}{\geqslant}

\newcommand{\aaa}{\mathcal{A}}
\newcommand{\bb}{\mathcal{B}}
\newcommand{\cc}{\mathcal{C}}
\newcommand{\dd}{\mathcal{D}}
\newcommand{\ddk}{\mathcal{D}^{(k-1)}}
\newcommand{\ee}{\mathcal{E}}
\newcommand{\ff}{\mathcal{F}}
\newcommand{\ddkf}{\ddk(\ff)}
\newcommand{\fg}{\mathcal{G}}

\newcommand{\hh}{\mathcal{H}}

\newcommand{\E}{\mathsf{E}}

\newcommand{\Prb}{\mathsf{P}}

\newtheorem{thm}{Theorem}

\newtheorem{claim}[thm]{Claim}
\newtheorem{lemma}[thm]{Lemma}

\newtheorem{prop}[thm]{Proposition}
\newtheorem{obs}[thm]{Observation}
\newtheorem{conj}[thm]{Conjecture}

\theoremstyle{plain} 
\newcommand{\thistheoremname}{}
\newtheorem*{genericthm}{\thistheoremname}
\newenvironment{namedthm}[1]
  {\renewcommand{\thistheoremname}{#1}%
   \begin{genericthm}}
  {\end{genericthm}}

\title{Best possible bounds on the number of distinct differences in intersecting families}
\author{
Peter Frankl\footnote{R\'enyi Institute, Budapest, Hungary and Moscow Institute of Physics and Technology, Russia, Email: {\tt peter.frankl@gmail.com}},
Sergei Kiselev\footnote{Moscow Institute of Physics and Technology, Email: {\tt kiselev.sg@gmail.com}}, 
Andrey Kupavskii\footnote{G-SCOP, CNRS, University Grenoble-Alpes, France and Moscow Institute of Physics and Technology, Russia; Email: {\tt kupavskii@yandex.ru}. Research of the author is supported by the grant
RSF N 21-71-10092.}}
\begin{document}

\maketitle

\begin{abstract}
    For a family $\mathcal F$, let $\mathcal D(\mathcal F)$ stand for the family of all sets that can be expressed as $F\setminus G$, where $F,G\in \mathcal F$. A family $\mathcal F$ is intersecting if any two sets from the family have non-empty intersection. In this paper, we study the following question: what is the maximum of $|\mathcal D(\mathcal F)|$ for  an intersecting  family of $k$-element sets? Frankl conjectured that the maximum is attained when $\mathcal F$ is the family of all sets containing a fixed element. We show that this holds if  $n \ge 50k\ln k$ and $k \ge 50$. At the same time, we provide a counterexample for $n< 4k$.
\end{abstract}
\section{Introduction}

Let $n \ge k \ge 1$ be integers and let $[n] = \{1, \ldots, n\}$ be the standard $n$-element set. Let further $\binom{[n]}{k}$ denote the collection of all $k$-element subsets ($k$-sets) of $[n]$. A subset $\ff$ of $\binom{[n]}{k}$ is called an \emph{intersecting family} if $F\cap F' \ne\varnothing$ for all $F,F'\in\ff$. 
Since for $n < 2k$, $F\cap F'\ne\varnothing$ is always true, from now on we always assume that $n \ge 2k$. Let us recall the Erd\H os-Ko-Rado Theorem, one of the fundamental results in extremal set theory.
\begin{thm}[\cite{EKR}]
Suppose that $\ff\subset\binom{[n]}{k}$ is intersecting and $n \ge 2k$. Then
\begin{equation}
\label{EKR_bound}
    |\ff| \le \binom{n-1}{k-1}.
\end{equation}
\end{thm}

For a fixed element $x\in[n]$ one defines the \emph{full star} $\mathcal S_x := \left\{F\in \binom{[n]}{k}\colon x\in F\right\}$. Subfamilies of $\mathcal S_x$ are called \emph{stars}.

Full stars provide equality in \eqref{EKR_bound}. On the other hand Hilton and Milner \cite{HM} proved that for $n > 2k$ no other intersecting family attains equality in \eqref{EKR_bound}.

\begin{thm}[\cite{HM}]
Suppose that $\ff\subset\binom{[n]}{k}$, $\ff$ is intersecting but $\ff$ is not a star. If $n > 2k$ then
\begin{equation}
\label{HM_bound}
|\ff| \le \binom{n-1}{k-1} - \binom{n-k-1}{k-1} + 1.
\end{equation}
\end{thm}

For a family $\ff$ define the family of (setwise) differences: $\dd(\ff) := \{F\setminus F'\colon F,F'\in\ff\}$. For $0 \le \ell\le k$ define also $\dd^{(\ell)}(\ff) := \{ D\in\dd(\ff)\colon |D| = \ell\}$. Note that $\ff\subset\binom{[n]}{k}$ is intersecting if and only if $\dd^{(k)}(\ff) = \varnothing$. We also note that one of the first correlation inequalities is the Marica--Sch\"onheim Theorem \cite{MS}, which states that $|\dd(\ff)|\ge |\ff|$ for any family $\ff\subset 2^{[n]}.$ 
\begin{obs}
Suppose that $\ff\subset\binom{[n]}{k}$ is intersecting. Then
$$
|\dd(\ff)| = \sum_{0\le\ell<k} |\dd^{(\ell)}(\ff)| \le
\sum_{0\le\ell<k} \binom{n}{\ell}.
$$
On the other hand,
$$
|\dd(\mathcal S_x)| = \sum_{0\le\ell < k} \binom{n-1}{\ell}.
$$
\end{obs}

In \cite{F_diff}, the following conjecture is stated. 
\begin{conj}[\cite{F_diff}]
\label{F_diff_conj}
Suppose that $n > 2k > 0$, $\ff\subset\binom{[n]}{k}$ is intersecting. Then
\begin{equation}
\label{main_thm_sum}
|\dd(\ff)|
\le \sum_{0\le\ell<k} \binom{n-1}{\ell}.
\end{equation}
\end{conj}
In \cite{F_diff} the conjecture was proved for $n \ge k(k+3)$. The main result of this paper is the following theorem.
\begin{thm}
\label{main_thm}
Conjecture~\ref{F_diff_conj} is true for $n \ge 50k\ln k$, $k \ge 50$. Moreover, equality in \eqref{main_thm_sum} is attained only for full stars.
\end{thm}
We also provide a counterexample for the conjecture when $n$ is close to $2k$.
\begin{thm}
\label{theorem: counterexample}
For integers $n, k$, where $n = ck$, $2<c<4$ and $k > k(c)$ sufficiently large, there is an intersecting family $\ff$ with $|\dd(\ff)|$ greater than the right hand side of \eqref{main_thm_sum}.
\end{thm}

Let us outline the approach  to prove Theorem~\ref{main_thm}.
Since every $D\in\ddkf$ is a subset of some $F\in\ff$, we have
$$
|\ddkf| \le k\cdot |\ff|.
$$

One can use this easy bound to deduce the conjecture for $\ff$ with $|\ff|$ small. 
\begin{prop}\label{obs7}
Let $\ff\subset\binom{[n]}{k}$ be intersecting. Suppose that $|\ff| \le \binom{n-1}{k-1} / (2k)$, $n \ge 3k$. Then
$$
|\dd(\ff)| \le
\sum_{0\le\ell<k} \binom{n-1}{\ell}.
$$
\end{prop}
We will prove Proposition~\ref{obs7} in the beginning of Section~\ref{section: proof of small div}.

The difficult case is how to deal with relatively large intersecting families. There are various structural theorems concerning relatively large intersecting families, e.g. \cite{DF, Frankl1987EKR, KZ}. However, the smaller is $n$ w.r.t. $k$, the less we can infer. For $n < k^2$ the bound in \eqref{HM_bound} is larger than $\binom{n-1}{k-1}/3$. For $n = o(k^2)$ there are various intersecting families that have, say, very large covering number, but still contain $(1 - o(1))\binom{n-1}{k-1}$ members. Even the junta approximation of Dinur and Friedgut gives results that are strong enough to combine with Proposition~\ref{obs7}, say, only for $n>Ck^{1+\eps},$ where $\eps>0$ is some fixed constant and $C = C(\eps)$ is a very large constant depending on $\eps.$  The approach that we found is to apply a powerful concentration inequality of the third author (cf. \cite{FK_conc}) to prove that for $n > ck\ln k$ the lower bound $|\ff| > \binom{n-1}{k-1}/(2k)$ is sufficient to guarantee that the overwhelming majority of the members of $\ff$ contain a fixed vertex. To state this result let us recall the definition of $\gamma(\ff)$, the diversity of the family $\ff$.
$$
\gamma(\ff) := \min_x |\{F\in\ff\colon x\notin F\}|.
$$
Clearly, $\gamma(\ff) = 0$ iff $\ff$ is a star. Then Theorem~\ref{main_thm} is implied by the the following two lemmas.
\begin{lemma}
\label{lemma: small div}
Let $\ff\subset\binom{[n]}{k}$ be intersecting. Suppose that $n \ge 50k\ln k$, $k \ge 50$ and \eqref{main_thm_sum} does not hold. Then
$$
\gamma(\ff) \le 
\binom{n -\ln k}{n - k - 1}.
$$
\end{lemma}
\begin{lemma}
\label{lemma: from small div}
Let $n, k$ be positive integers satisfying $n \ge 50k\ln k$, $k \ge 50$.
Suppose that $\ff\subset \binom{[n]}{k}$ is an intersecting family with
$$
1 \le \gamma(\ff) \le \binom{n-\ln k}{n-k-1},
$$
Then
$$
|\dd(\ff)| < \sum_{0\le i< k}\binom{n-1}{i}.
$$
\end{lemma}

\medskip

The paper is organised as follows. In Section~\ref{section: notation} we introduce the necessary notation and list some standard tools that we use. 
In Section~\ref{section: main thm from small diversity} we give a purely combinatorial proof of Lemma~\ref{lemma: from small div}. 
In Section~\ref{section: concentration and small div} we state the concentration inequality and apply it to prove Lemma~\ref{lemma: small div}.
In Section~\ref{section: countrexample} we prove Theorem~\ref{theorem: counterexample}.
The last section contains some remarks and open problems.

\section{Notation and tools}
\label{section: notation}

For a family $\ff\subset 2^{[n]}$ and a pair of subsets $A\subset B\subset[n]$ we use the standard notation $$\ff(A, B) := \{F\setminus A\colon F\in\ff, F\cap B = A\}.$$

In the cases $A = B$ or $A = \varnothing$ we use the shorter form $\ff(B) := \ff(B, B)$, $\ff(\overline B) := \ff(\varnothing, B)$. If $B = \{x\}$, then we set $\ff(x) := \ff(\{x\})$ and $\ff(\bar x) := \ff(\overline{ \{x\} })$.

For $0 \le i \le k$ and $\ff\subset\binom{[n]}{k}$ we define the \emph{$i$-th level shadow}
$\partial^{i} \ff := \{S\in\binom{[n]}{i}\colon \exists F\in\ff, S\subset F\}$.




For $k$ and $\ff$ fixed Kruskal \cite{Kruskal1963number} and Katona \cite{Katona1966theorem} determined the minimum of $|\partial^i \ff|$ for all $0 \le i < k$. The following handy version of the Kruskal--Katona theorem is due to Lov\'asz \cite{Lovasz1979combinatorial}.
\begin{namedthm}{Kruskal-Katona Theorem}[Lov\'asz version]
Suppose that $\ff\subset\binom{[n]}{k}$ and $x$ is the unique real number satisfying $k \le x \le n$, $|\ff| = \binom{x}{k}$. Then
\begin{equation}
\label{KK_Lovasz}
|\partial^i \ff| \ge \binom{x}{i}, \quad 0 \le i < k.
\end{equation}
\end{namedthm}
(For a short proof of \eqref{KK_Lovasz}, c.f. \cite{Frankl1984short}.)

Two families $\ff\subset\binom{[n]}{k}$, $\fg\subset\binom{[n]}{\ell}$ are called \emph{cross-intersecting} if $F\cap G \ne\varnothing$ for all $F\in\ff$, $G\in\fg$.
For a family $\ff\subset\binom{[n]}{l}$ let $\ff^c$ denote the family of complements, $\ff^c := \{[n]\setminus F\colon F\in\ff\}$.

Katona \cite{Katona1964intersection} was the first to note that for $n \ge k + \ell$, $\ff$ and $\fg$ are cross-intersecting iff $\ff \cap \partial^k(\fg^c) = \varnothing$. Together with the Kruskal-Katona Theorem, this immediately implies 

\begin{thm}
\label{crossint}
Let $m \ge a + b$. If $\mathcal A \subset \binom{[m]}{a}$ and $\mathcal B \subset \binom{[m]}{b}$ are cross-intersecting and $|\mathcal A| \ge \binom{x}{m-a}$ for some $m - a \le x \le m$, then $|\mathcal B| \le \binom{m}{b} - \binom{x}{b}$.
\end{thm}

\section{Families with small diversity}
\label{section: main thm from small diversity}
In this section, we prove Lemma~\ref{lemma: from small div}. W.l.o.g. we assume that element 1 has  maximal degree in $\ff$, i.e., that $\gamma(\ff) = |\ff(\bar 1)|$.

For notational convenience, set $\fg := \{[2, n] \setminus F\colon F\in \ff(\bar 1)\}$ and note that $\fg\subset\binom{[2,n]}{n-k-1}$. Choose $x$ so that $|\fg| = {x\choose n-k-1}$ and note that $x\le n-\ln k$ by the assumption. Put $r:=\ln k$ for shorthand. Using  \eqref{KK_Lovasz}, we get
\begin{align}
\frac{|\fg|}{|\partial^{k-1} \fg|} 
&\le 
\frac{{x\choose n-k-1}}{{x\choose k-1}}=
\prod_{i=k}^{n-k-1}\frac{x-i+1}{i} \le 
\prod_{i=k}^{n-k-1}\frac{n-r-i+1}{i} = 
\prod_{i=k}^{n-k-1}\frac{i-r+2}{i}\notag \\ 
&=
\prod_{i=k}^{n-k-1}\Big(1-\frac{r-2}{i}\Big)\le 
{\rm exp}\Big(-(r-2)\sum_{i=k}^{n-k-1}\frac{1}{i}\Big)\le 
{\rm exp}\Big(-(r-2)\ln\frac{n-k}{k}\Big) \notag 
\\
& \label{small_div_ineq1} \le 
e^{-(\ln k - 2)\ln (50\ln k - 1)} 
< k^{-5.27} \cdot (50\ln k - 1)^2
= k^{-2} \cdot \frac{(50\ln k - 1)^2}{k^{3.27}}
<
k^{-2}.
\end{align}

Set $\dd := \dd(\ff)$. In order to give an upper bound for $|\dd|$, we use $|\dd| = |\dd(1)| + |\dd(\bar 1)|$ and distinguish $D\in\dd$ according to their size. We will need the following two inequalities.
\begin{equation}
\label{ineq_3.4}
\sum_{0\le i\le k-2} |\dd^{(i)}(\bar 1)| \le 
\sum_{0\le i\le k-2} \binom{n-1}{i},
\end{equation}
\begin{equation}
\label{ineq_3.5}
|\partial^{k-1} \ff(\bar 1)|\le 
k |\ff(\bar 1)| = k|\fg|  \overset{\eqref{small_div_ineq1}}{<} |\partial^{k-1} \fg| / k.
\end{equation}

The inequality \eqref{ineq_3.4} is true simply because $\dd^{(i)}(\bar 1)\subset {[2,n]\choose i}$.  The inequality \eqref{ineq_3.5} uses a trivial upper bound on the size of the $(k-1)$-shadow and inequality \eqref{small_div_ineq1}. Next, we prove the following inequality.
\begin{equation}
\label{ineq_3.6}
|\dd^{(k-1)}(\ff(1))| \le \binom{n-1}{k-1} - |\partial^{k-1} \fg|.
\end{equation}
\begin{proof}
Suppose that $H\in\dd^{(k-1)}(\ff(1))$. Then $H\in\binom{[2,n]}{k-1}$ and $\{1\}\cup H\in \ff$. To prove \eqref{ineq_3.6}, we show $H\notin \partial^{k-1}\fg$. Indeed, the opposite would mean $H \subset [2,n]\setminus F$ for some $F\in\ff(\bar 1)$. However this would imply $(\{1\}\cup H) \cap F = \varnothing$, a contradiction.
\end{proof}

Note that the members $D\in\dd^{(k-1)}(\bar 1)$ that are not accounted for in the left hand side of \eqref{ineq_3.6} are of the form $F\setminus F'$ with $F \in \ff(\bar 1), F'\in \ff$. Consequently, each such $D$ is a member of $\partial^{k-1}\ff(\bar 1)$. Combining this with \eqref{ineq_3.5} and \eqref{ineq_3.6}, we infer
$$
|\dd^{(k-1)}(\bar 1)| < \binom{n-1}{k-1} - \frac{k-1}{k} |\partial^{k-1}\fg|.
$$
Adding \eqref{ineq_3.4} to it:
\begin{equation}
\label{ineq_3.7}
|\dd(\bar 1)| < \sum_{0\le i< k} \binom{n-1}{i} - \frac{k-1}{k}|\partial^{k-1}\fg|.    
\end{equation}

In order to bound $\dd(1)$, first note that $H\in\dd(1)$ means that $\{1\}\cup H = F\setminus F'$ for $F,F'\in\ff$, where $1\in F$, $1\notin F'$. Since $H\cap F' = \varnothing$, $H\in \partial^{|H|}\fg$. Also, $F\cap F' \ne \varnothing$ implies $0 \le |H| \le k-2$. Consequently,
$$
|\dd(1)| \le
\sum_{0\le i\le k-2} |\partial^i \fg|.
$$

Let us prove that
\begin{equation}
\label{ineq_3.9}
|\partial^{i-1} \fg| < \frac18 |\partial^i \fg|\quad \text{ for } 1 \le i \le k-1.
\end{equation}
\begin{proof}
This follows from the local LYM inequality, but for completeness we give a direct proof.
To this end, consider the bipartite graph with parts $\partial^{i-1} \fg$ and $\partial^i \fg$, where an edge is  $(A, B)$ is drawn iff $A\subset B$. Let $e$ be the number of edges in this graph. Then $e = i|\partial^i \fg|$ is obvious. On the other hand, if $A\in \partial^{i-1} \fg$ then for some $(n-k-1)$-set $G\in\fg$ we have $A\subset G$. Consequently, there are at least $n-k-i$ edges with  $A$ as one of the endpoints. Thus, $e \ge (n-k-i) |\partial^{i-1} \fg|$. Using $n-k-i > 8k$ and $i < k$, \eqref{ineq_3.9} follows.
\end{proof}

From \eqref{ineq_3.9} we deduce
\begin{equation}
\label{ineq_3.10}
|\dd(1)| \le \sum_{0\le i\le k-2} |\partial^i\fg| <
\frac17 |\partial^{k-1} \fg|.
\end{equation}

Combining \eqref{ineq_3.7} and \eqref{ineq_3.10}, the statement of the lemma follows for $k \ge 2$. The case $k=1$ is obvious, and thus the proof is complete.

\section{The diversity is small}
\label{section: concentration and small div}
We start this section with the aforementioned concentration result. We believe that it may be of good use in many other settings in extremal set theory. 
\subsection{Concentration}

The following result states that, informally, for $k$-uniform families of density $\alpha$ the densities of the sizes of their intersections with large subsets of the ground set are tightly concentrated around $\alpha$.
\begin{thm}
\label{complement_conc}
Fix integers $m, \ell, \ell', t$ such that $m \ge t\ell + \ell'$, fix 
$a > 0$ and set $\eps = \frac{2a + \sqrt{8\ln 2}}{\sqrt t}$.
Let $\fg \subset \binom{[m]}{\ell}$ be a family and set $\alpha := |\fg| / \binom{m}{\ell}$. Let $H$ be  distributed uniformly on $\binom{m}{\ell'}$. Then
\begin{equation}\label{eqconc}
\Prb\left[ |\fg({\overline H})| < (\alpha - \eps) \binom{m - \ell'}{\ell} \right]
< 2 e^{-a^2/2}.
\end{equation}
\end{thm}
We note that the same bound for the probability holds for the upper deviations. 
The proof of this result relies on the following concentration result by Kupavskii.
\begin{thm}[\cite{FK_conc}]
\label{mathcings_conc}
Fix integers $m, l, t$ such that $m \ge tl$. Let $\fg \subset \binom{[m]}{l}$ be a family and set $\alpha := |\fg| / \binom{m}{l}$. Let $\mathcal B$ be chosen uniformly at random out of all $t$-matchings of $\ell$-sets and let $\eta$ be a random variable $|\fg \cap \mathcal M|$. Then $\E[\eta] = \alpha t$ and, for any positive $a$ and $\delta\in\{-1, 1\}$, we have
$$
\Prb[\delta\cdot (\eta - \alpha t) \ge 2a \sqrt{t}] \le e^{-a^2/2}.
$$
\end{thm}

\begin{proof}[Proof of Theorem~\ref{complement_conc}]
Let $\hh\subset\binom{m}{\ell'}$ be the family of sets $H$ such that $|\fg({\overline H})| < (\alpha - \eps) \binom{m - \ell'}{\ell}$. 
Denote $\beta := |\hh| / \binom{m}{\ell'}$ and assume that $\beta \ge 2 e^{-a^2/2}$.
Let $(H, B_1, \ldots, B_t)$, where $|H| = \ell'$ and $|B_i| = \ell$ for $i\in[t]$, be a  $(t+1)$-matching  chosen uniformly at random out of all matchings of sets with such sizes. Note that $H$ is distributed uniformly on $\binom{[m]}{\ell'}$ and, in particular, $\Prb[H\in\hh] = \beta$. Moreover, the subset of the matching above $\bb := (B_1, \ldots, B_t)$ is distributed uniformly on the set of all $t$-matchings of $\ell$-element subsets of $[m]$, and by Theorem~\ref{mathcings_conc} we have
$$
\Prb\big[|\bb \cap \fg| < \alpha t - 2a \sqrt{t}\big] < e^{-a^2/2} \le \beta / 2.
$$
Therefore, we have $\Prb\big[|\bb \cap \fg| < \alpha t - 2a \sqrt{t} \mid H\in\hh\big] < 1/2$, and, in particular, there is a set $H'\in\mathcal H$ such that 
\begin{equation}
\label{smol_mathcing_prob}
\Prb[|\bb \cap \fg| < \alpha t - 2a \sqrt{t} \mid H=H'] < 1/2.
\end{equation}

Fix such a set $H'$ and put $\fg' := \fg(\overline{H'})$, $\alpha' := |\fg'| / \binom{m-\ell'}{\ell} < \alpha - \eps$. Denote a uniformly random $t$-matching of $\ell$-subsets in $[m]\setminus H'$  by $\bb'$. Note that for a fixed $H=H'$ the random matching $\bb$ is distributed the same way as $\bb'$.  Then~\eqref{smol_mathcing_prob} can be written as
\begin{equation}\label{eqfirstevent}
\Prb[|\bb' \cap \fg'| < \alpha t - 2a \sqrt{t} ] < 1/2.
\end{equation}
On the other hand, from Theorem~\ref{mathcings_conc} we have
\begin{equation}\label{eqsecondevent}
\Prb[|\bb' \cap \fg'| > \alpha' t + 2\sqrt{2 t\ln 2}] \le e^{-\ln 2} = 1/2.
\end{equation}
With positive probability, neither of the events in the left hand sides of \eqref{eqfirstevent} and \eqref{eqsecondevent} hold.  This implies
$\alpha t - 2a \sqrt{t} < \alpha' t + 2\sqrt{2 t\ln 2} < (\alpha - \eps) t +  \sqrt{8t\ln 2}$, which is a contradiction with the definition of $\eps$.
\end{proof}

\subsection{Proof of Lemma~\ref{lemma: small div}}
\label{section: proof of small div}

Let $\ff\subset\binom{[n]}{k}$ be an intersecting family that is not a star and that does not satisfy \eqref{main_thm_sum}. Denote $C := n / k$.

Note that 
$$
\sum_{i=1}^{k-2} \binom{n-1}{i-1}\le
2{n-1\choose k-3} = 
\frac {2(k-1)(k-2)}{(n-k+1)(n-k+2)}{n-1\choose k-1} <
\frac{2}{(C-1)^2}{n-1\choose k-1}.
$$
Therefore, if  \eqref{main_thm_sum} does not hold for $\ff$, we get that 
$$
|\ddkf| >
\sum_{i=0}^{k-1} \binom{n-1}{i} - \sum_{i=0}^{k-2} \binom{n}{i} =
\binom{n-1}{k-1} - \sum_{i=1}^{k-2} \binom{n-1}{i-1} \ge
\left(1 - \frac{2}{(C-1)^2}\right)\binom{n-1}{k-1}.
$$
Since $|\ddkf| \le k|\ff|$, we also have $|\ff| \ge \frac{1}{2k}\binom{n-1}{k-1}$ for $C\ge 3$, which proves Proposition~\ref{obs7}. 
Now we can bound density of $\ddkf$:
$$
|\ddkf| / \binom{n}{k-1} \ge
\left(1 - \frac{2}{(C-1)^2}\right) \frac{n-k+1}{n} \ge
\left(1 - \frac{2}{(C-1)^2}\right) \left(1 - \frac{1}{C}\right) \ge
1- \frac{2}{C}.
$$
\medskip
We apply Theorem~\ref{complement_conc} for $\ddkf$ with $m = n$, $\ell = k - 1$, $\ell' = k$, $t=\lfloor C - 1\rfloor$, $a = \sqrt{2\ln 8n}$, $\eps = \frac{2a+\sqrt{8\ln 2}}{\sqrt{t}}$. 
Note that 
\begin{align*}
\eps & < 
\frac{2\sqrt{2\ln (8Ck)} + 2.4}{\sqrt{C - 2}} < 
\frac{2\sqrt{2\ln (400k\ln k)} + 2.4}{\sqrt{50\ln k - 2}} < 
2\sqrt{ \frac{2}{49} \left(1 + \frac{\ln 400}{\ln k} + \frac{\ln\ln k}{\ln k}\right) } + 0.18
\\
& < 
2\sqrt{ \frac{2}{49} (1 + 1.54 + 0.35) } + 0.18 <
0.87 < 
0.9 - \frac{2}{C}.
\end{align*}
We conclude that the probability in the right hand side of \eqref{eqconc} is at most $2e^{-a^2/2} = 2e^{-\ln 8n} = \frac 1{4n}.$ Therefore, the property in \eqref{eqconc} cannot be satisfied for all $H \in \ff$ (indeed, the probability that a randomly chosen $H$ belongs to $\ff$ is at least $\frac 1{2n}$). Therefore, unveiling what \eqref{eqconc} states with these parameters, we get that there is $F\in\ff$ such that $|\ddkf \cap \binom{\overline F}{k-1}| \ge \frac{1}{10} \binom{n-k}{k-1}$. W.l.o.g. we assume that $F = \{1, \ldots, k\}$.

By the definition of $\ddkf$ and since $\ff$ is intersecting, for each $D \in \ddkf \cap \binom{\overline F}{k-1}$ there is a $x\in F$ such that $\{x\}\sqcup D\in\ff$.
Split $\ddkf \cap \binom{\overline F}{k-1}$ into $\bigsqcup_{x\in F} \ee_x$ such that for any $D\in\ee_x$ we have $\{x\}\sqcup D\in\ff$. Note that this partition is not unique and the families $\ee_x$ are pairwise cross-intersecting.

Put $m = n - k$, $\ell = k - 1$ and $r :=  \ln k $.

\begin{claim}
\label{big_or_smol}
For some $i$ we have 
$|\ee_i| \ge \binom{m}{\ell} - \binom{m-r}{\ell}$.
\end{claim}
\begin{proof}
Assume that there is no $i$ such that $|\ee_i|>|\ddkf \cap \binom{\overline F}{k-1}|-{m-1\choose \ell-1}.$ Then we can partition $F = I_1\sqcup I_2$ such that $|\cup_{i\in I_1} \ee_i|>{m-1\choose \ell-1}$ and $|\cup_{i\in I_2} \ee_i|>{m-1\choose \ell-1}$. But the families $\cup_{i\in I_1} \ee_i$ and $\cup_{i\in I_2} \ee_i$ are cross-intersecting, which contradicts Theorem~\ref{crossint} with $a = b = \ell$ and $x = m-1.$

Therefore, we can assume that $|\ee_i|>|\ddkf \cap \binom{\overline F}{k-1}|-{m-1\choose \ell-1}$. We are only left to verify that this quantity is at least the bound stated in the claim. Note that 
$$
|\ee_i| > 
\Big|\ddkf \cap \binom{\overline F}{k-1}\Big| - {m-1\choose \ell-1} > 
\left(\frac{1}{10} - \frac \ell m\right){m\choose \ell} > 
\frac{1}{12}{m\choose \ell}.
$$

In order to prove the claim, we are thus only left to show that ${m-r\choose \ell} > \frac{11}{12} {m\choose \ell}.$ Indeed, noting that $m-\ell-r+1 = C\ell + C - 2\ell - r > (C-2)\ell$, we have

\begin{small}$$
\binom{m}{\ell} / \binom{m-r}{\ell} \le
\left(\frac {m-\ell+1}{m-r-\ell+1}\right)^\ell \le
\exp \left(\frac{\ell r}{m-r-\ell+1} \right) <
\exp\left(\frac{r}{C-2}\right) < 
\exp\left(\frac{1}{20}\right) <
\frac{12}{11}.
$$\end{small}
\end{proof}

In what follows, we w.l.o.g. assume that the $i$ guaranteed by Claim~\ref{big_or_smol}  is equal to $1$.

\begin{claim}
\label{claim_2}
Let $A$ be such that $A \subset F$, $1\notin A$. Then $|\ff(A, F)| \le \binom{m - r}{m - k + |A|}$.
\end{claim}
\begin{proof}
Note that $\ff(A,F)\subset {[n]\setminus F\choose k-|A|}$. Since $\ee_1$ and $\ff(A, F)$ are cross-intersecting and $|\ee_1| > \binom{m}{\ell} - \binom{m-r}{\ell}$, from Theorem~\ref{crossint} we get $|\ff(A, F)| \le \binom{m - r}{m - k + |A|}$.
\end{proof}

\begin{claim} We have
\label{smol_diversity}
$$
|\ff(\bar 1)| \le \binom{n-r-1}{n-k-1}.
$$
\end{claim}
\begin{proof}
Decomposing and applying Claim~\ref{claim_2} we get
$$
|\ff(\bar 1)| = \sum_{A\subset F, 1\notin A} |\ff(A, F)| \le 
\sum_{a=0}^{k-1} \binom{k-1}{k-1-a}\binom{m - r}{m-k+a} =
\binom{n-r-1}{n-k-1}
$$
(in the last inequality we use that $m = n - k$).
\end{proof}
Lemma~\ref{lemma: small div} follows from Claim~\ref{smol_diversity} immediately.

\section{Counterexamples}
\label{section: countrexample}

For $2 \le p \le k$, $n > 2k$, define the family $\aaa_p(n, k)$ by
$$
\aaa_p(n, k) := 
\left\{A\in\binom{[n]}{k}\colon 1\in A, A\cap[2, p+1] \ne\varnothing\right\} \cup 
\left\{A\in\binom{[n]}{k}\colon [2, p+1]\subset A\right\}.
$$
Note that $\aaa_k(n, k)$ is the only family that attains equality in the Hilton--Milner Theorem~\eqref{HM_bound} for~$k > 3$.
In \cite{Frankl1987EKR} it was shown that families $\aaa_p(n, k)$ are extremal among intersecting families with fixed maximum degree. In \cite{KZ} the stronger result showing that they are extremal among intersecting families with fixed minimal diversity is proven. Also let us note that $|\aaa_2(n, k)| = |\aaa_3(n, k)|$.

In the following, we show that both $\aaa_3(n,k)$ and $\aaa_k(n,k)$ are counterexamples to Conjecture~\ref{F_diff_conj} for $n = ck$ and small $c$. We have $|\dd(\aaa_3(n,k))|>|\dd(\mathcal S_x)|$ for $2<c<\frac 12(3+\sqrt 5)$ and $|\dd(\aaa_k(n,k))|>|\dd(\mathcal S_x)|$ for $2<c<4$. At the same time, it is not difficult to see that for $2<c<\frac 12(3+\sqrt 5)$ we have $|\dd(\aaa_3(n,k))|>|\dd(\aaa_k(n,k))|$. 

We believe that each $\aaa_t(n,k)$ for $4\le t\le k$ provides a counterexample to the conjecture for some values of $c>2$, but it requires some tedious computations, so we have not checked it.

\subsection{$\aaa_k(n,k)$}

Let us put $\bb:= \aaa_k(n, k)$ for shorthand. 
Comparing it with the full star $\mathcal S_1 = \{A\in\binom{[n]}{k}\colon 1\in A\}$ we have 
$\bb = \left( \mathcal S_1
\setminus \left\{F\cup\{1\}\colon F\in\binom{[k+2,n]}{k - 1}\right\}\right) \cup \{[2, k+1]\}$.
This easily implies
\begin{equation}
\label{eq: countrex_1}
\dd(\mathcal S_1)\setminus\dd(\bb) = \binom{[k+2,n]}{k-1},
\end{equation}
i.e., we ``lose'' $\binom{n-k-1}{k-1}$ sets.

On the other hand for all $D\subset [k+2,n]$, $0\le |D| \le k-2$, the sets $\{1\}\cup D$ is in $\dd(\bb)\setminus\dd(\mathcal S_1)$. Indeed, fix $E\subset [2, k+1]$, $|E| = k-1-|D|$, Then $\{1\} \cup D\cup E \in \bb$ and $(\{1\}\cup D\cup E)\setminus[2, k+1] = \{1\}\cup D$. Thus
\begin{equation}
\label{eq: countrex_2}
|\dd(\bb)\setminus \dd(\mathcal S_1)| \ge 
\sum_{\ell=0}^{k-2} \binom{n-k-1}{\ell}.
\end{equation}

As long as the RHS of~\eqref{eq: countrex_1} is smaller than the RHS of~\eqref{eq: countrex_2}, $|\dd(\bb)| > |\dd(\mathcal S_1)|$, i.e., we get a counterexample.

\begin{claim}
If $2 < c < 4$ and $n = ck$, $k > k_0(c)$, then
\begin{equation}
\label{eq: counterex claim}
\sum_{\ell=0}^{k-2} \binom{n-k-1}{\ell} > \binom{n-k-1}{k-1}.
\end{equation}
\end{claim}
\begin{proof}
Take $d = c / 2$, $1 < d < 2$. Define $t$ as the minimum integer satisfying
\begin{equation}
\label{eq: counterex_3}
\frac{1}{d} + \frac{1}{d^2} + \ldots + \frac{1}{d^t} \ge 1.
\end{equation}

We want to choose $k_0 = k_0(c)$ to satisfy
\begin{equation}
\label{eq: counterex_4}
\binom{n-k-1}{k-1-j} > d^{-j} \binom{n-k-1}{k-1}\quad\mbox{ for }1 \le j \le t.
\end{equation}
Combined with \eqref{eq: counterex_3} this would clearly imply~\eqref{eq: counterex claim}. For~\eqref{eq: counterex_4} it is sufficient that
$$
\binom{n-k-1}{k-1-j} / \binom{n-k-1}{k-j} = \frac{k-j}{n-2k+j} \ge \frac{1}{d} = \frac{2}{c}
\quad\mbox{ for }1 \le j \le t.
$$
Equivalently, $(4 + c)k \ge 2n + (2 + c)j$ or $k \ge \frac{2 + c}{4 - c}\cdot t$. Thus, for sufficiently large $k$ the inequalities \eqref{eq: counterex_4} hold and therefore the claim is proved. 
\end{proof}

\subsection{$\aaa_3(n,k)$}

Put $\cc := \aaa_3(n, k)$.
For a set $F\in\binom{[n]}{k}$ its membership in $\cc$ is decided by the intersection $F\cap [4]$, it is a so-called ``junta''. Define
$$
\mathcal J^* := \{\{1,2\}, \{1,3\}, \{1,4\}, \{1,2,3\}, \{1,2,4\}, \{1,3,4\}, \{2,3,4\}, \{1,2,3,4\}\},
$$
then $F\in\cc$ iff $F\cap [4] \in \mathcal J^*$. (We say that $\mathcal J^*$ is the {\it defining family} of $\mathcal J$.) Note that $\mathcal J^*$ is intersecting.
We need to show that
$$
|\dd(\cc)| > \sum_{i=1}^{k} \binom{n-1}{k-i}.
$$

First put $\mathcal L_i := \{F\setminus F'\colon F,F'\in\mathcal J^*, |F\cap F'| \le i\}$. Note that by the definition $\varnothing = \mathcal L_0\subset\mathcal L_1 \subset \mathcal L_2 \subset \mathcal L_3 \subset \mathcal L_4 = \mathcal L_5 = \ldots = \mathcal L_k$.

Now we want to rewrite $|\dd^{(k-i)}(\cc)|$ using $\mathcal L_i$. Note that for each $F\setminus F'\in \dd^{(k-i)}(\cc)$ we have $|F\cap F'| = i$, thus for $J := F\cap [4]$ and $J' := F'\cap [4]$ we have $|J\cap J'| \le i$ and $J\setminus J' \in \mathcal L_i$. This implies that $\dd^{(k-i)}(\cc)\subset \{F\in {n\choose k-i}: F\cap [4]\in \mathcal L_i\}$. Actually, it is not difficult to see that equality holds in the previous inclusion. Grouping the size of the family in the right hand side by its intersections with $[4]$, we get
$$
|\dd^{(k-i)}(\cc)| = \sum_{b=0}^4  |\mathcal L_i^{(b)}| \binom{n-4}{k-i-b}.
$$
Summing this over $i = 1, \ldots, k$ we get
\begin{equation}
\label{eq: diff size decomposition}
|\dd(\cc)| =
\sum_{i=1}^k |\dd^{(k-i)}(\cc)| =
\sum_{i=1}^k \sum_{b=0}^4 |\mathcal L_i^{(b)}| \binom{n-4}{k-i-b}.
\end{equation}

For our $\mathcal J^*$ we have $\mathcal L_1 = \{\{1\}, \{2\}, \{3\}, \{4\}, \{2,3\}, \{2,4\}, \{3,4\}\}$ and $\mathcal L_b = \mathcal L_1 \cup \{\varnothing\}$ for $b > 1$. Then we can regroup terms and rewrite
\begin{align*}
|\dd(\cc)| 
& =
\sum_{i=1}^k \left( |\mathcal L_{i+1}^{(0)}| \binom{n-4}{k-i-1} + \sum_{b=1}^2 |\mathcal L_i^{(b)}| \binom{n-4}{k-i-b} \right)
\\
& =
\sum_{i=1}^k \left( \binom{n-4}{k-i-1} +4\binom{n-4}{k-i-1} +3\cdot \binom{n-4}{k-i-2} \right)\\
& = 
5\sum_{i=0}^{k-2}{n-4\choose i}+3\sum_{i=0}^{k-3}{n-4\choose i}
\end{align*}

Note that we can use Pascal's triangle identity several times and rewrite 
$$
|\dd(\mathcal S_1)| = 
\sum_{i=0}^{k-1} \binom{n-1}{i} =
\sum_{i=0}^{k-1} {n-4\choose i}+3\sum_{i=0}^{k-2}{n-4\choose i}+3\sum_{i=0}^{k-3}{n-4\choose i}+\sum_{i=0}^{k-4}{n-4\choose i}.
$$

Therefore, we have
$$|\dd(\mathcal S_1)|-|\dd(\cc)| = {n-4\choose k-1}-{n-4\choose k-2}-{n-4\choose k-3} = {n-4\choose k-1}-{n-3\choose k-2}.$$
The latter expression is negative iff $\frac 1{k-1}<\frac{n-3}{(n-k-1)(n-k-2)},$ which is equivalent to $(n-k-1)(n-k-2)<(n-3)(k-1).$ Simplifying, we get $n^2-(3k+2)n+k^2+6k-1<0$, which holds for $n<\frac 12(3k+2+\sqrt{5k^2-12k+8})\sim \frac 12(3+\sqrt 5)k$ for $k\to \infty.$


\medskip

\textbf{Remark.} Put  $\alpha := 1/c$ and $\mu_\alpha(F) := \alpha^{|F|}(1-\alpha)^{|\mathcal J^*|-|F|}$, $\mu_\alpha(\cc):=\sum_{F\in\cc} \mu_\alpha(F),$ and $\mathcal L_{=i} := \mathcal L_{i} - \mathcal L_{i-1}$. Then we can rewrite \eqref{eq: diff size decomposition} as follows.
$$
|\dd(\cc)| \sim
\sum_{i=1}^k \binom{n}{k-i} \mu_\alpha(\mathcal L_i) =
\sum_{i=1}^k \sum_{s=1}^{|\mathcal J^*|} \binom{n}{k-i+1-s} \mu_\alpha(\mathcal L_{=s}) \sim
$$
$$
\sim
\sum_{i=1}^k \binom{n}{k-i} \sum_{s=1}^{|\mathcal J^*|} \left(\frac{\alpha}{1-\alpha}\right)^{s-1} \mu_\alpha(\mathcal L_{=s}).
$$
Thus the problem of maximizing $|\dd(\cc)|$ for an intersecting junta can be reduced to a problem of maximizing $\sum_{s=1}^{|\mathcal J^*|} \left(\frac{\alpha}{1-\alpha}\right)^{s-1} \mu_\alpha(\mathcal L_{=s})$ for the defining family of the junta.

\section{Conclusion}

It is natural to ask similar questions for intersecting families, where 
set difference operation is replaced with another binary set operation. E.g. define $\mathcal{SD}(\ff) := \{F\Delta F'\colon F, F'\in\ff\}$, where $\Delta$ stands for symmetric difference. It is clear that for a full star $\mathcal S_x$ we have
$$
|\mathcal{SD}(\mathcal S_x)| = \sum_{0\le\ell<k} \binom{n-1}{2\ell}.
$$
\begin{conj}\label{conjsd}
For any intersecting family $\ff$, $n > 10k$, we have
\end{conj}
\begin{equation}
\label{eq: conj sym diff}
|\mathcal{SD}(\ff)| \le \sum_{0\le\ell<k} \binom{n-1}{2\ell}.
\end{equation}

Using an argument analogous to \cite{F_diff}, one can show that~\eqref{eq: conj sym diff} holds for, say, $n>3k^2$. Repeating the argument from Section~\ref{section: proof of small div} one can show that for intersecting $\ff\subset{[n]\choose k}$ with $|\ff| > n^{-100} \binom{n}{k}$ that maximizes $|\mathcal{SD}(\ff)|$ we have $\gamma(\ff) < n^{-100}\binom{n}{k}$ for $n > Ck\ln k$.

However, the general case is elusive. It is even unclear how to show that $|\ff|$ should be large if $\ff$ is extremal. It is not difficult to see that a $p$-random subset of  a full star $\mathcal S_x$ in a wide range of parameters has almost exactly the same number of symmetric differences as $\mathcal S_x$.  We believe that proving Conjecture~\ref{conjsd} even for $n>100 k\ln k$ would already be very interesting.


\medskip

Another question one may ask is as follows: what is the maximum of $\mu_p(\dd(\ff))$ for intersecting $\ff\subset 2^{[n]}$ and $0 < p < \frac 12$? In \cite{F_diff} the first author proved that $\mu_{1/2}(\dd(\ff)) \le \frac12$. 
To solve this problem we can essentially repeat the argument from \cite{F_diff}. It was proved in \cite{EKR} that for any intersecting family $\ff\subset 2^{[n]}$ there is an intersecting family $\fg$, such that $\ff\subset\fg\subset 2^{[n]}$, $|\fg| = 2^{n-1}$.
Using the notation $\overline X := [n]\setminus X$, it is clear that for any $X\subset [n]$, either $X$ or $\overline X$ belongs to $\fg$ and $\fg$ is upwards closed. Next, if $X = F\setminus F'$, where $F,F'\in \fg$, then $F'\subset \overline X$, and so $\overline X\in \fg.$ 
We immediately get that $X\notin \fg.$ 
Conversely, if $X\in \fg$ then $[n]\in\fg$ implies $\overline X\in \mathcal D(\fg).$ 
Therefore, we have $\fg \sqcup \dd(\fg) = 2^{[n]}$ and thus $\mu_p(\fg) + \mu_p(\dd(\fg)) = 1$. 
This shows that $\mu_p(\dd(\fg))$ is maximized whenever $\mu_p(\fg)$ is minimized. Since $|\{X,\overline X\}\cap \fg| = 1$, in order to minimize the measure of $\fg$ it is better to include the larger set in $\fg.$ The optimal family also happens to be the intersecting family, corresponding to the majority function (if $n$ is even, then out of all $n/2$-element sets we can take those that contain $1$, say). An interesting consequence is that the maximum of $\mu_p(\mathcal D(\ff))=1-e^{-cn},$ where $c>0$ for any $p<\frac 12.$ This is in contrast with the uniform case, in which the answer more or less corresponds to the $p$-measure $1-p$.

\begin{small}

\end{small}

\end{document}